\documentclass[11pt,leqno]{amsart}%
\usepackage{amsmath}
\usepackage{mathptmx}
\usepackage{amsfonts}
\usepackage{amssymb}
\usepackage{graphicx}%
\usepackage{color}

\providecommand{\U}[1]{\protect\rule{.1in}{.1in}}
\newtheorem{theorem}{Theorem}

\newtheorem{corollary}[theorem]{Corollary}

\newtheorem{lemma}[theorem]{Lemma}

\newtheorem{remark}[theorem]{Remark}

\newcommand{\diver}{{{\rm{div}\,}}}

\begin{document}

\title[Spectral, stochastic and  curvature  estimates]{Spectral, stochastic and  curvature  estimates  for submanifolds of highly negative
curved  spaces}
\author{G. Pacelli  Bessa}
\address{Departamento de Matem\'atica \\Universidade Federal do Cear\'a-UFC\\
60455-760 Fortaleza, CE, Brazil}
\email{bessa@mat.ufc.br}
\author{Stefano Pigola}
\address{Sezione di Matematica - DiSAT\\
Universit\'a dell'Insubria - Como\\
via Valleggio 11\\
I-22100 Como, ITALY}
\email{stefano.pigola{@}uninsubria.it}
\author{Alberto G. Setti}
\address{Sezione di Matematica - DiSAT\\
Universit\'a dell'Insubria - Como\\
via Valleggio 11\\
I-22100 Como, ITALY}
\email{alberto.setti@uninsubria.it}
\thanks{The third named author is grateful to the Universidade Federal do
Cear\'a-UFC for the very warm hospitality and to CAPES-Ci\^{e}ncia Sem Fronteiras, grant
400794/2012-8, for financial support.}
\begin{abstract}We prove spectral, stochastic and mean curvature estimates for complete $m$-submanifolds $\varphi \colon M \to N$ of $n$-manifolds with a pole $N$   in terms of the comparison isoperimetric ratio $I_{m}$ and the extrinsic radius $r_\varphi\leq \infty$.
Our proof holds for the bounded case $r_\varphi< \infty$, recovering the known results, as well as for the unbounded case $r_{\varphi}=\infty$. In both cases, the fundamental ingredient in these estimates is the integrability over $(0, r_\varphi)$ of the inverse
$I_{m}^{-1}$ of the comparison isoperimetric radius. When $r_{\varphi}=\infty$,  this condition is guaranteed if $N$ is highly negatively curved.
\end{abstract}
\maketitle

\section{Introduction }
Let $N$ be a complete Riemannian $n$-manifold with a pole $p$ and let   $\varphi \colon M \to N$ be an isometric immersion  of a complete Riemannian $m$-manifold $M$ into  $N$. The extrinsic radius $r_{\varphi}$ is the (possibly extended) number $$r_{\varphi}=\inf\{ r\in (0, \infty]\colon\, \varphi (M)\subset B_{N}(r)\}, $$where $B_{N}(r)\subseteq N$ is the geodesic ball  centered at $p$ of $N$ with radius $r\in(0, \infty]$. Suppose that the radial\footnote{Along the geodesics issuing from $p$.} sectional curvatures of $N$ at $x\in B_{N}(r_{\varphi})$ satisfies
\begin{equation}
\label{eqCurv}K_N(x)\leq - G(\rho_{N}(x))
\end{equation}
where $G\colon\mathbb{R}\rightarrow\mathbb{R}$ is a smooth
even function and $\rho_{N}(x)={\rm dist}_{N}(p,x)$.

 Associate to $N$, the  $m$-dimensional {\em model} manifold
$\mathbb{M}_{\sigma}^{m}$ with
radial sectional curvature $-G\left(  r\right)  $. Namely,%
\[
\mathbb{M}_{\sigma}^{m}=\left(  [0,r_{\varphi}]  \times\mathbb{S}%
^{m-1}\!\!,\,\,ds^{2}_{\sigma}=dr^2+\sigma^{2}\left(r\right)d\theta^{2}\right),
\]
where $\sigma$ denotes the unique solution of the Cauchy problem on $(0, r_{\varphi}]$
\begin{equation}
\label{sigma ivp}
\left\{
\begin{array}{rll}
\begin{array}{l}
   \sigma''(t)-G(t)\sigma(t)=0, \\[0.1cm]
   \sigma(0)=0, \,\,\,\,\; \sigma'(0)=1,
  \end{array}
\end{array}\right.
\end{equation}
which we assume to be positive and increasing on $[0,r_\varphi)$. Observe that if $G\geq 0$ on $[0,\infty)$ then
$\sigma''\geq 0$ everywhere and  $\sigma'\geq 1>0$ on $[0,\infty).$ More generally,
it follows by a strengthened version of Kneser Theorem, see e.g.
\cite[Prop. 1.21]{bianchini-mari-rigoli}, that if
$G_{-}=\max\{ -G, 0\}$ satisfies
\begin{equation} t\int_{t}^{\infty}G_{-}(s)ds
\leq \frac{1}{4},
\label{eqBMR-memoirs}
\end{equation}
then $\sigma'\geq 0$.

Associate to the model  $\mathbb{M}_{\sigma}^{m}$
 the function $I_{m}\colon [0, r_{\varphi}]\to \mathbb{R}_{+}$ defined by
\[
I_{m}\left(  r\right)  =\frac{\sigma\left(  r\right)  ^{m-1}}{\int_{0}^{r}%
\sigma\left(  t\right)  ^{m-1}dt}.
\]
 In geometric terms,  $I_{m}$ is the
\textit{non-homogeneous isoperimetric ratio}%
\[
I_{m}\left(  r\right)  =\frac{\mathrm{vol\,}\partial B_{r}\left(  o\right)
}{\mathrm{vol\,}B_{r}\left(  o\right)  },
\]
where $B_{r}\left(  o\right)  $ and $\partial B_{r}\left(  o\right)  $ are the
geodesic balls and spheres in $\mathbb{M}_{\sigma}^{m}$ of radius $r>0$ and
center at the pole $o$ of the model. In particular, the Cheeger constant
$h\left(  \mathbb{M}_{\sigma}^{m}\right)  $ of the model has  the upper
estimate%
\[
\inf_{\lbrack0,+\infty)}I_{m}\left(  r\right)  \geq h\left(
\mathbb{M}_{\sigma}^{m}\right)  .
\]
Also associated to the model  $\mathbb{M}_{\sigma}^{m}$, is the \textit{homogeneous isoperimetric ratio} $\mathcal{I}_{m}\left(  r\right)$
which is defined by
\[
\mathcal{I}_{m}\left(  r\right)  =\frac{\sigma\left(  r\right)  ^{m}}{\int_{0}%
^{r}\sigma\left(  t\right)  ^{m-1}dt}=\frac{\mathrm{vol\,}\partial B_{r}\left(
o\right)^{\frac{m}{m-1}}}{\mathrm{vol\,}B_{r}\left(  o\right)  }.
\]

We showed in \cite{BPS-POTA} that  if $\varphi\colon M \to N$ is minimal and  the following extrinsic conditions
$$ \mathcal{I}_{m}'(t)\geq 0,\,\,t\in (0, r_{\varphi}) \quad \text{and}\quad I_m^{-1}\in L^{1}\left(0, r_{\varphi}\right)
$$
hold, then  the global mean exit time  of $ M$ is finite, and, in fact,
$$
E_{M}\leq\smallint_{0}^{ r_{\varphi}}I_{m}^{-1}(s)ds< \infty.
$$
In particular, $M$ is not $L^{1}$-Liouville.
Three aspects of this result  should be remarked. First, the condition
$ \mathcal{I}_{m}'(t)\geq 0$ is implied by $-G \leq 0$.
 Indeed, $$\mathcal{I}_{m}'(t)=
m\frac{\sigma'}{\sigma} - \frac{\sigma^{m-1}}{\int_0^t \sigma^{m-1}}=\frac{1}{\sigma \int_0^t \sigma^{m-1}} \left( m\sigma' \int_0^t \sigma^{m-1} - \sigma^m\right), $$ thus $\mathcal{I}_{m}'(t)\geq 0 \Leftrightarrow
\left( m\sigma' \int_0^t \sigma^{m-1} - \sigma^m\right)\geq 0$. Now,  $\left( m\sigma' \int_0^t \sigma^{m-1} - \sigma^m\right)\to 0$ as $r\to 0+$ and its derivative is
$$
m\sigma''\int_0^t \sigma^{m-1}.
$$ Thus if $\sigma''\geq 0$, that is, if   the curvature is nonpositive, then $m\sigma''\int_0^t \sigma^{m-1}\geq 0$ and therefore $\mathcal{I}_{m}'(t)\geq 0$.
Note however that, while it is necessary that $\sigma''\geq 0$ in a right neighborhood of $0$, $\mathcal{I}_{m}(t)$ could be nondecreasing even in the presence of some controlled negativity of $\sigma''$.

Second, the requirement   $I_m^{-1}\in L^{1}\left(0, r_{\varphi}\right) $ is automatically satisfied  if $ r_{\varphi}<\infty$, and in this case,  in \cite[Thm.9]{BPS-POTA} we recover
S. Markvorsen's result \cite[Thm.1-item ii.]{markvorsen-JDG} under the slightly weaker hypothesis $\mathcal{I}_{m}'(t)\geq 0$.

Finally, in the case where the extrinsic diameter is infinite, the condition that $I_{m}^{-1}$ is integrable is equivalent to the stochastic incompleteness and implies a great amount of negative curvature of the $m$-dimensional model $\mathbb{M}^{m}_{\sigma}$  (and thus of $N$).

The result  \cite[Thm.9]{BPS-POTA} suggests that there should exist a correspondence between   results  valid for  complete, bounded  submanifolds of $ N$  and companion results for complete, unbounded immersions $\varphi \colon M \to N$, into a manifold  $N$ with a pole, with radial sectional curvatures bounded above as in \eqref{eqCurv} and  such that $I_m^{-1}\in L^{1}\left(0, +\infty \right) $.

The purpose of this paper is to show that   this correspondence does exist for a variety of   results, including  well known curvature,   stochastic and spectral estimates for bounded submanifolds, of which we shall prove counterparts in the unbounded highly negatively curved setting.

\goodbreak
\section{Statement of the results}
\begin{theorem} [Spectral estimates]
\label{spec-est}
Let $\varphi\colon M\rightarrow N$ be an isometric
immersion of a complete $m$-dimensional Riemannian manifold $M  $ into the complete $n$-dimensional Riemannian manifold $N$ with a
pole $p\in N$. Let $\rho_{N}\left(  y\right)  ={\rm dist}_{N}(p,y)  $ and \label{SpectralEstimates}
assume that the radial sectional curvature of $N$ satisfies%
\[
\mathrm{Sec}_{\mathrm{rad}}^{N}\left(  x\right)  \leq-G\left(  \rho_{N}\left(
x\right)  \right)
\]
for some smooth, even function $G$. Assume also that the solution $\sigma$ of
\eqref{sigma ivp} satisfies $\sigma'\geq 0$ on $[0,\mathrm{diam}\varphi(M)]$, that  $\mathcal{I}_m$ is nondecreasing in that interval, and
\[
A=\sup_{M}\displaystyle\frac{\left\vert \mathbf{H}\right\vert \left(  x\right)  }{I_{m}\left(
\rho_{N}\left(  \varphi\left(  x\right)  \right)  \right)  }\leq 1,\]
where $\mathbf{H}$ is the mean curvature vector field of $\varphi$.
Then
\begin{enumerate}
\item[(a)] The bottom of the spectrum of the Laplace-Beltrami operator of $M$
satisfies the estimate%
\[
\lambda^{\ast}\left(M\right)  \geq
\max\left\{\frac{\left(  1-A\right)  }
{\int_{0}^{r_{\varphi}  }I_{m}\left(  t\right)  ^{-1}dt},
\,\, (1-A)^2 \inf_{[0,r_{\varphi}  ]}
\frac{I_{m}\left(  r\right)
^{2}}{4}\right\}
\]
\item[(b)]
If $\varphi$ is proper in the geodesic ball $B_{N}(r_{\varphi})$, $\int_0^{r_{\varphi}} I_m^{-1} dt< +\infty$ and $A<1$ then the spectrum of $-\Delta_{M}$ is discrete.
\end{enumerate}
\end{theorem}
Observe that item (b) of Theorem \ref{spec-est} extends the main result of \cite{bjm-JGA}.
It is worth noticing that in the case where $\sigma'/\sigma$ is nonincreasing, then
 $I_m'(r)<0$.
Indeed, it is easy to check that $I_m$ satisfies the
Riccati equation
\[
I_m' + I_m^2 = (m-1)\frac{\sigma'}{\sigma}I_m,
\]
and that $I_m(r)\sim m/r$ as $r\to 0$.
It follows that $I_m'<0$ and $I_m> (m-1)\sigma'/\sigma$ in a right neighborhood
of $0$, and an easy comparison argument shows that if the right hand side of the
Riccati  equation is nonincreasing then $I_m'<0$ where defined. In particular,
\[
\inf_{[0,r_{\varphi}]} I_m = \lim_{r\to r_{\varphi}} I_m(r)
\geq \lim_{r\to r_\varphi} (m-1)\frac{\sigma'}{\sigma}.
\]
In the special case where $K_N(x)\leq -k^2<0$, so that
$\mathbb{M}_\sigma^m = \mathbb{H}_{-k^2}^m$ is
hyperbolic space of curvature $-k^2$,
we have
\[
\inf_{[0,r_\varphi)}I_m =
\begin{cases}
(m-1)k \coth (kr_\varphi) &\text{if } r_\varphi<+\infty,\\
(m-1)k &\text{if } r_\varphi = +\infty
\end{cases}
\]
Thus, if $r_\varphi= +\infty$ and  $|\mathbf{H}|\leq H_o < (m-1)k$, we have
\[
(1-A)^2\inf_{[0,+\infty)}
\frac{I_{m}(r)^{2}}{4}= \left(1-\frac{H_o}{(m-1)k}\right)^2\frac{(m-1)^2k^2}4 =
\frac{[(m-1)k-H_o]^2}4,
\]
and, in particular, we recover a result by L.-F. Cheung and P.-F. Leung,
\cite[Theorem 2]{CheungLeung-MathZ}, and Bessa  and Montenegro,
\cite[Corollary 4.4]{BeMo-angloban}.
Similarly, in the case where  $r_\varphi<+\infty$, if $k>0$  $H_o <(m-1)k \coth(kr_\varphi)$,
then
\[
(1-A)^2\inf_{[0,+\infty)}
\frac{I_{m}(r)^{2}}{4}=
\frac{[(m-1)k\coth(kr_\varphi) -H_o]^2}4,
\]
while if $k=0$ and and $H_o <(m-1)/r_\varphi$, then
\[
(1-A)^2\inf_{[0,+\infty)}
\frac{I_{m}(r)^{2}}{4}=
\frac{[(m-1)/r_\varphi -H_o]^2}4.
\]
and we recover  results by K. Seo \cite{Seo-Monat}.

A closer inspection of the proof of Theorem~\ref{spec-est} shows that if we let $M=N$,
then the conclusion holds without having to assume that $\mathcal{I}_n$ be increasing.
Thus we have
\begin{corollary}
\label{cor1}
Let $N$ be a complete Riemannian $n$-manifold with a pole $p$ and  radial sectional curvature satisfying
$$
\mathrm{Sec}_{\mathrm{rad}}^{N}\left(  x\right) \leq - G(\rho_{N}(x))
$$
where $G\colon\mathbb{R}\rightarrow\mathbb{R}$ is a smooth
even function and the solution $\sigma$ of  the initial value problem
\eqref{sigma ivp} satisfies $\sigma' \geq 0$ on $[0,+\infty)$.
Then
\begin{equation}
\lambda^{\ast}(N)\geq \lambda^{\ast}(\mathbb{M}_{\sigma}^{n})\geq \max\left\{\displaystyle\frac{\left[\inf_{(0,\infty)}I_{n}(r)\right]^{2}}{4},\,
\frac{1}{\int_{0}^{\infty}I_{n}^{-1}(r)dr} \right\}\cdot
\end{equation}
Moreover, if $\int_0^{+\infty} I_n^{-1} dt< +\infty$, then the spectrum of $N$ is purely discrete.
\end{corollary}
Observe that both alternatives do occur:
\[
\frac{1}{\int_{0}^{r_{\varphi}  }
I_m\left(  t\right) ^{-1}dt} >
\inf_{[0, r_{\varphi})}\frac{I_{m}\left(  r\right)^{2} }{4}
\]
or
\[
\frac{1}{\int_{0}^{r_{\varphi}  }I_m\left(  t\right)^{-1}dt}
< \inf_{[0, r_{\varphi})}\frac{I_{m}\left(  r\right)^{2} }{4}
\]
as shown by the examples below.
Indeed, consider the $2$-dimensional model $\mathbb{M}^{2}_{\sigma}$, with
$$
\sigma(t)=(r + \displaystyle\frac{r^{7}}{2})\displaystyle
\exp{\displaystyle\frac{r^{6}}{6}}.
$$
It is easy to show that
$$
\displaystyle\frac{1}{\int_{0}^{\infty}I_{2}(s)^{-1}ds}
=\frac{3\cdot 2^{2/3} \sqrt{3}}{\pi}\approx 2.62>
\displaystyle\inf_{[0, \infty)}\frac{ I_{2}(r)^2}{4}=\frac{36}{25}\cdot
\left(\frac{2}{3}\right)^{-1/3}\approx 1.64.
$$

On the other hand, if $\mathbb{M}^{m}_{\sigma}=\mathbb{H}^{m}(-1)$
is a totally geodesic hyperbolic space in $N=\mathbb{H}^{n}(-1)$,
then
$$
\displaystyle\frac{1}{\int_{0}^{\infty}I_{m}(s)^{-1}ds}=0
$$
by stochastic completeness, while, as observed above,
$$
\displaystyle\inf_{[0, \infty)} \frac{I_{m}(r)^2}{4}
=\frac{(m-1)}{4}.
$$

We next describe mean curvature estimates which extend previous results valid for
bounded immersions obtained, in increasing generality, in \cite{Aminov-MatSb},
\cite{HasanisKoutroufiotis-ArchMath},  \cite{JorgeXavier-MathZ}, \cite{Karp-MathAnn},
and \cite{PRS-PAMS}.
\begin{theorem}[Mean curvature estimates]
\label{thmMCE}
Let $\varphi\colon \!M\rightarrow N$ be an  isometric
immersion of a stochastically complete, $m$-dimensional Riemannian manifold
$M  $ into a $n$-dimensional Riemannian manifold
$N$ with a pole $o\in N$. Assume that the radial sectional curvature of $N$ satisfies%
\[
\mathrm{Sec}_{\mathrm{rad}}^{N}\left(  x\right)  \leq-G\left(  \rho_{N}\left(
x\right)  \right)
\]
for some smooth, even function $G\left(  t\right) $ satisfying \eqref{eqBMR-memoirs}. Assume also that
\begin{enumerate}
\item[(i)] $\mathcal{I}_{m}\left(  r\right)  $ is non-decreasing for $r\in (0, r_{\varphi}]$

\item[(ii)]
\[
{I_m\left(  r\right)^{-1}  }\in L^{1}\left(0, r_{\varphi}\right)  .
\]
\end{enumerate}
Then
\[
\sup_M \frac{\vert\mathbf{H}(x)\vert}{I_{m}(\rho_x)}\geq 1
\]
In particular, if  $r_{\varphi}=+\infty$, ($\varphi$ is unbounded in $N$) then
\[
\limsup_{x\to \infty} \frac{\vert\mathbf{H}(x)\vert}{I_{m}(\rho_x)}\geq 1.
\]
If $r_{\varphi}=+\infty$ and $I_m(r)^{-1}\to 0$ as $r\to +\infty$  then
\[
\sup_M |H| = +\infty.
\]\end{theorem}
\begin{remark}
If the mean curvature $H$ of  $M$ is bounded, then either $r_\varphi <+\infty$, and $\varphi$  has bounded image in $N$, or   $I_m(r)^{-1}\not\to 0$ as $r\to +\infty$. An immediate consequence  is that, under the above assumptions,
an immersed submanifold with bounded mean curvature in $N$ is
stochastically incomplete. This completes the picture initiated in
 \cite[Thm.10]{BPS-POTA}, where we proved that, regardless the condition (ii) on the isoperimetric ratio $I_{m}\left(  r\right)  $, a properly immersed minimal
submanifold of $N$ is not $L^{1}$-Liouville (hence stochastically incomplete).

We also note that, according to Theorem \ref{thmMCE}, Theorem \ref{spec-est} does not apply if $I_{m}^{-1}$ is integrable and $M$ is stochastically complete.
\end{remark}
\begin{remark} An application of de L'Hospital rule show that condition $I_m(r)^{-1}\to 0$ as $r\to \infty$ holds provided $\displaystyle\frac{\sigma'}{\sigma}\to +\infty$, which in turn is typical of a super-exponential behavior of $\sigma(r)$. We have already remarked that  condition $\mathcal{I}_{m}'\geq 0$ is satisfied provided $\sigma''(r)\geq 0$.
Finally, as already recalled, the integrability of $I_m(r)^{-1}$ is equivalent to the stochastic incompleteness of the model $\mathbb{M}^m_\sigma$ and is implied by a sufficiently fast growth of $\sigma$.
\end{remark}

\section{Preliminaries}Let $M$ be a smooth  Riemannian manifold and  $\Omega \subset M$ an
arbitrary open subset. The fundamental tone
$\lambda^{\ast}(\Omega )$ of $\Omega$, is defined by
$$
\lambda^{\ast}(\Omega )=
\inf\left\{\frac{\int_{\Omega}\vert \nabla f\vert^{2}}{\int_{\Omega} f^{2}},\,f\in
 H^1_0(\Omega )\setminus\{0\}\right\},
$$
where $H_{0}^{1}(\Omega )$ is the completion  of $C^{\infty}_{0}(\Omega )$ with respect to
the norm
$$
\Vert\varphi \Vert_{\Omega}^2=\int_{\Omega}\varphi^{2}+\int_{\Omega} \vert\nabla \varphi\vert^2.
$$
When $\Omega=M$ is a complete non-compact   Riemannian manifold,
the fundamental tone $\lambda^{\ast}(M)$ coincides with the bottom   $\inf \Sigma(-\Delta)$
of the $L^2$-spectrum $\Sigma(-\Delta) \subset [0,\infty)$ of the unique self-adjoint extension of the Laplacian
$\Delta$ acting on $C_{0}^{\infty}(M)$ also denoted by
$\Delta$.
When $\Omega$ is compact with  boundary $\partial \Omega$, then the fundamental tone
is the bottom of the $L^2$-spectrum of the Friedrichs
extension of  $-\Delta$ initially defined on $C_c^\infty (\Omega)$.
Moreover,  there exists $u\in C^{\infty}(\Omega)\cap H_{0}^{1}(\Omega)$, positive in $\Omega$
satisfying $\triangle u+\lambda^{\ast}(\Omega)u=0$, ($u\vert_{\partial \Omega} =0$
if $\partial \Omega\neq \emptyset$ and piecewise smooth).
The spectrum decomposes as $\Sigma(-\Delta)=\Sigma_{p}(-\Delta)\cup\Sigma_{ess}(-\Delta)$
where $\Sigma_{p}(-\Delta)$ is  formed by   eigenvalues  with finite multiplicity
and $\Sigma_{ess}(-\Delta)$ is  formed by accumulation points of the spectrum and
by the eigenvalues  with infinite multiplicity.
It is said that  $M$ has discrete spectrum if $\Sigma_{ess}(-\Delta)=\emptyset$
and that  $M$ has purely continuous spectrum if  $\Sigma_{p}(-\Delta)=\emptyset$.
It is well known that for every exhaustion  of $M$ by relatively compact open sets
$\{K_j\}$ with  boundary, one has,
$
\inf \Sigma_{ess}(-\Delta) = \lim_{j \to +\infty} \lambda^*(M \backslash K_j).
$
It follows that $-\Delta$ has pure discrete spectrum if and only if
$$
\lim_{j \to +\infty} \lambda^*(M \backslash K_j)=\infty.
$$
The following two lemmas are useful to  obtain lower bounds for the fundamental tones
of open sets of Riemannian manifolds.
\begin{lemma}[\cite{BeMo-angloban}]
Let $\Omega \subset M$ be an open subset of  a Riemannian manifold $M$.
Then the fundamental tone of $\Omega$ is bounded below by
\begin{equation}\lambda^{\ast}(\Omega)\geq \frac{c(\Omega)^{2}}{4},\label{eqML1}
\end{equation}
\label{mainlemma}
where
$c(\Omega)=\sup\left\{\displaystyle\frac{\inf_{\Omega}  \diver X}
{\sup_{\Omega}\vert X\vert}:\, X\in{\mathcal X}^{\infty}(\Omega ),\, \diver X \geq 0\right\}
$
and
$ {\mathcal X}^{\infty}(\Omega )$ is the set of all smooth vector fields in $\Omega$.
\end{lemma}
\begin{lemma}[Barta \cite{barta}, \cite{BeMo-angloban2}]
Let $\Omega \subset M$ be an open subset of  a Riemannian manifold $M$ and
$u\colon \Omega \to \mathbb{R}$ be a smooth positive function.
Then
$$\lambda^{\ast}(\Omega) \geq \inf_{\Omega}\left[-\frac{\Delta u}{u}\right]\cdot
$$
\end{lemma}

\section{Proof of the results}
\begin{proof}[Proof of Theorem \ref{SpectralEstimates}]
Recall that if $\varphi \colon M\to N$ is an isometric immersion  and $g:N\to
\mathbb{R}$ and $F:\mathbb{R}\to \mathbb{R}$ are smooth functions, then for
every $X\in T_xM$ we have
\[
\begin{split}
\mathrm{Hess}(F\circ g\circ \varphi) (X,X) &=
F''(g(\varphi(x))) \langle \nabla^N g, d\varphi X\rangle^2\\
&+F'(g(\varphi(x))) \bigl[
\mathrm{Hess}^N g (d\varphi X,d\varphi X) +\langle \nabla^N g , II
(X,X)\rangle
\bigr].
\end{split}
\]
If  $g = \rho_{N}$ is the distance function, then the assumption
on the sectional curvature of $N$ implies
\[
\mathrm{Hess}_{N} \rho_{N} (Y,Y) \geq \frac{\sigma'}\sigma \bigl[\langle Y, Y\rangle -
\langle \nabla^N \rho_{N}, Y\rangle^2\bigr]
\]
on $B_N(r_\varphi)$.

Assuming that $F'\geq 0$, letting $\{X_i\}_{i=1}^m$ be an orthonormal basis of $T_x M$ and setting $\rho_{x}=\rho_{N}(\varphi(x)) $, we obtain
\begin{eqnarray}
\Delta (F\circ \rho\circ \varphi)(x) &\geq &  m \bigl (F'\frac {\sigma'}\sigma\bigr)(\rho_{x})+
\bigl(F''-F'\frac{\sigma'}\sigma \bigr)(\rho_{x}) \sum_{i=1}^m \langle
\nabla^N \rho_{N}, d\varphi X_i\rangle^2,\nonumber \\
&& +\,F'(\rho_{x})\langle\nabla^N \rho_{N}, \mathbf{H}\rangle \nonumber.
\end{eqnarray}
Let
\begin{equation}
\label{F}
F\left( r\right)  =\int_{0}^{r}\frac{\int_{0}^{t}\sigma^{m-1}\left(
s\right)  ds}{\sigma^{m-1}\left(  t\right)  }dt,
\end{equation}
so that
$$
F' (r) =
\frac{\int_{0}^{r}\sigma^{m-1}\left(  s\right)
ds}{\sigma^{m-1}\left(  r\right)  }= I_m^{-1}(r)>0\,\,\,
{\rm and}\,\,\,
\
F''(r) = 1 - (m-1) \frac {\sigma'}\sigma F'(r),
$$
which, inserted into the last inequality, yield
\begin{eqnarray}
\Delta (F\circ \rho\circ \varphi)(x)& \geq & m \bigl (F'\frac
{\sigma'}\sigma\bigr)(\rho_x)+ \bigl[(1- m\frac{\sigma'}\sigma
F_R'(\rho_x)\bigr] \sum_{i=1}^m \langle
\nabla^N \rho, d\varphi X_i\rangle^2\nonumber \\
& & -F'(\rho_x)\vert H(\varphi (x))\vert\, \nonumber
\end{eqnarray}
We complete $\{d\varphi X_i\}_{i=1}^m$ to an orthonormal basis
$\{d\varphi X_i\}_{i=1}^m\cup \{Y_j\}_{j=m+1}^n$ on $T_{\varphi(x)}N$,  and
note that
\[
\sum_{i} \langle \nabla^N\rho, d\varphi X_i\rangle^2 + \sum_j \langle \nabla^N\rho,
Y_j\rangle^2 =1.
\]
Inserting  this into the above inequality and using the assumption $ \mathcal{I}_{m}'(t)\geq 0$ in the form
\[
m\frac{\sigma'}\sigma F'=m\frac{\sigma'}\sigma\frac{\int_{0}^{t}\sigma^{m-1}\left(  s\right)
ds}{\sigma^{m-1}\left(  t\right)  } \geq 1
\]
we finally obtain
\begin{eqnarray}\label{eq-H}
\Delta (F\circ \rho\circ \varphi)(x)& \geq & 1 + \left[ m \bigl (F'\frac
{\sigma'}\sigma\bigr)(\rho_x)-1 \right]\sum_j \langle \nabla^N\rho,
Y_j\rangle^2 - F'(\rho_x)\vert {\mathbf H}(\varphi(x))\vert\nonumber\\
&\geq & 1 - F'(\rho_x)\vert {\mathbf H}(\varphi(x))\vert.
\end{eqnarray}
Thus, if $X=\nabla(F\circ \rho\circ \varphi)$, we have
\begin{equation*}
\mathrm{div}_M X \geq  1- \sup_M \left(F'(\rho_x)  |\mathbf{H}(x)|\right) = 1-A
\end{equation*}
and
\begin{equation*}
|X|\leq F'(\rho_x) = I_m^{-1}(\rho_x) \leq \frac 1{\inf_{[0,r_{\varphi}]} I_m(r)},
\end{equation*}
and we conclude that
\[
\lambda^*(M) \geq (1-A)^2 \inf_{[0,r_{\varphi}]}\frac{\left( I_m(r)\right)^2} 4.
\]
The estimate
\[
\lambda^*(M) \geq \frac{1-A}{ \int_0^{r_{\varphi}} I_m^{-1}(t) dt}
\]
in (a) is  an application of Barta's Theorem.
We consider first the case where $r_{\varphi}= +\infty$, and assume that
$I_m^{-1}\in L^1([0,+\infty))$ for otherwise the estimate is trivial.
Define
\[
\tilde{F}(r) =  \int_r^{+\infty} I_m^{-1}(t)dt,
\]
and let $u= \tilde{F}\circ \rho_N\circ \varphi$. Then $u$ is positive on $M$ and, since
\[
\tilde{F}'(r) = - I_m^{-1} (r) = -\frac{\int_0^r \sigma^{m-1} dt}{\sigma^{m-1}(r)} <0
\text{ and } \tilde{F}''(r) = -1 - (m-1) \frac{\sigma'}{\sigma} \tilde{F}'(r),
\]
a computation similar to that performed in the first part of the proof shows that \[
- \Delta_M u \geq 1 + |\mathbf{H}| \tilde{F}'(\rho_x)
= 1-\frac{|\mathbf{H}|}{I_m(r)}\geq 1-A
\]
and it follows from Barta's Theorem that
\[
\lambda^*(M)\geq \inf_M \left(\frac{-\Delta u}{u}\right) \geq \frac{1-A} {\int_0^{r_\varphi} I_m^{-1} dt},
\]
as required.

The case where $r_{\varphi}<+\infty$ is  similar. Since to apply Barta's theorem we need $u$ to be positive, we note that our assumptions imply that $I_m$ is well defined and positive in $[0,r_{\varphi}+\epsilon]$ for every $\epsilon>0$ sufficiently small. Next  we let
\[
\tilde{F}_{\epsilon}(r) = \int_r^{r_{\varphi}+ \epsilon} I_m^{-1} dt,
\]
and define $u_\epsilon$ accordingly. Arguing as above shows that
\[
\lambda^*(M) \geq \inf_M \left(\frac{-\Delta u_{\epsilon}}{u_{\epsilon}}\right)
\geq
\frac{1-A_{\epsilon}}{\int_0^{r_{\varphi}+ \epsilon} I_m^{-1} dt},
\]
 where
 \[
 A_{\epsilon} = \sup_M |\tilde{F}'_\epsilon(\rho_x) \mathbf{H}|.
 \]
 The conclusion now follows letting $\epsilon\to 0.$

Finally, if $\varphi$ is proper, $I_m^{-1}$ is integrable on $[0,\infty)$ and $A<1$, then the function $-u$ is bounded, proper, and satisfies
\[
\Delta (-u) \geq 1 - A >0
\]
on $M$.
Therefore, in the terminology of \cite{BPS-Revista}, it is a weak maximum principle violating exhaustion function, and the discreteness of the spectrum of $M$ follows from \cite[Theorem 32]{BPS-Revista}.
 \end{proof}
 \begin{proof}[Proof of Theorem \ref{thmMCE}] We maintain the notation of the first part of the proof of Theorem~\ref{spec-est}, and let $v=F\circ\rho\circ\varphi$. Then $v$ is bounded above by the assumption that $I_m^{-1}\in L^1([0,r_{\varphi}])$ and, by \eqref{eq-H}, it satisfies
 \[
 \Delta_M v \geq 1-\left(F'(\rho_x)  |\mathbf{H}(x)|\right) = 1- \frac{  |\mathbf{H}(x)|}{I_m(\rho_x)}.
 \]
Since $M$ is assumed to be stochastically complete, by the weak maximum principle at infinity there exists a  sequence $\{x_n\}$ in $M$ such that
\[
\lim_n v(x_n)= \sup_M v \quad \text{and}\quad \liminf_n \Delta_M v(x_n)\leq 0.
\]
Since $F$ is increasing, this implies that $\rho_{x_n}\to r_\varphi$. In particular, if $\varphi$ is unbounded, we conclude that
\[
\liminf_{x\to +\infty}\left(1- \frac{  |\mathbf{H}(x)|}{I_m(\rho_x)}\right) \leq 0,\,\,\text{ that is, } \,\, \limsup_{x\to \infty} \frac{  |\mathbf{H}(x)|}{I_m(\rho_x)}\geq 1.
\]
In the case where $\varphi$ is bounded, we can still conclude that
\[
\inf_M\left(1- \frac{  |\mathbf{H}(x)|}{I_m(\rho_x)}\right) \leq 0,
\,\,\text{
that is, }\,\,
\sup_M\frac{ |\mathbf{H}(x)|}{I_m(\rho_x)} \geq 1.
\]
\end{proof}

\section{Immersions into products}
The aim of this section is to prove versions of Theorems \ref{spec-est} and \ref{thmMCE} for immersions into a product manifold $N\times L$, where the factor $N$ satisfies condition \eqref{eqCurv} in the Introduction. This clearly relaxes the curvature conditions imposed on the target manifold.  As a counterpart, we need to strengthen the assumpions replacing conditions on  $\mathcal{I}_m$ and $I_M$ with  analogous conditions on $\mathcal{I}_{m-l}$ and $I_{m-l}$ respectively.

In some sense, we make up for the presence of the factor $L$ by imposing more negative curvature conditions on the factor $N$.

\begin{theorem}
\label{sp-est-products}
Let $\varphi$ be an isometric immersion of a complete Riemannian manifold $M  $ of dimension $m$ into the product $N\times \mathbb{R}^l$, where $L$ and $N$ are  complete Riemannian manifolds of dimension $n$  and $l$, respectively. Assume that $N$ has a pole  $p\in N$ and that its radial sectional curvature satisfies
\begin{equation}
\tag{\ref{eqCurv}}
K_N(x)\leq - G(\rho_{N}(x))
\end{equation}
where   $\rho_{N}\left(  y\right)  ={\rm dist}_{N}(p,y)  $
and  $G$ is a smooth, even function on $\mathbb{R}$. Assume that $m\geq l+1$ and that the solution $\sigma$ of  \eqref{sigma ivp} satisfies
$\sigma'\geq 0$ on $[0,r_{\pi_N\varphi}]$, where $\pi_N$ is the projection onto $N$.
Suppose further that $\mathcal{I}_{m-l}$ is nondecreasing in that interval,
and that
\[
A=\sup_{M}\displaystyle\frac{\left\vert \mathbf{H}\right\vert \left(  x\right)  }{I_{m-l}\left(
\rho_{N}\left( \pi_N \varphi\left(  x\right)  \right)  \right)  }\leq 1,
\]
where $\mathbf{H}$ is the mean curvature vector field of $\varphi$. Then
\begin{enumerate}
\item[(a)] The bottom of the spectrum of the Laplace-Beltrami operator of $M$
satisfies the estimate%
\[
\lambda^{\ast}\left(M\right)  \geq
\max\left\{
(1-A)^2 \inf_{[0,r_{\pi_N\varphi}]}
\frac{I_{m-l}\left(  r\right)
^{2}}{4},\,\,
\frac{\left(  1-A\right)  }
{\int_{0}^{r_{\pi_N\varphi} }I_{m-l}\left(  t\right)  ^{-1}dt}
\right\}
\]
\item[(b)]
If $\pi_N\circ \varphi$ is proper, $\int_0^{+\infty} I_{m-l}^{-1} dt< +\infty$ and $A<1$ then the spectrum of $-\Delta_{M}$ is discrete.
\end{enumerate}
\end{theorem}

\begin{proof}
We continue to keep the notation of the proof of Theorem~\ref{spec-est}.
Assuming that $F$ is smooth and satisfies $F'\geq 0,$ we consider the function $F\circ \rho_N \circ \pi_N \circ \varphi$,
and argue as in Theorem~\ref{spec-est} to obtain
\begin{eqnarray}
\Delta (F\circ \rho\circ\pi \circ \varphi)(x) &\geq &
\bigl(F''-F'\frac{\sigma'}\sigma \bigr)(\rho_{x}) \sum_{i=1}^m \langle
\nabla^N \rho_{N}, d\pi_N d\varphi X_i\rangle^2\nonumber \\
&& + \bigl (F'\frac {\sigma'}\sigma\bigr)(\rho_{x})\sum_{i=1}^m
|d\pi_N d\varphi(X_i)|^2 +\,F'(\rho_{x})\langle\nabla^N \rho_{N}, d\pi_N\mathbf{H}\rangle \nonumber.
\end{eqnarray}
Choosing
\begin{equation*}
F\left( r\right)  =\int_{0}^{r}\frac{\int_{0}^{t}\sigma^{m-l-2}\left(
s\right)  ds}{\sigma^{m-l-2}\left(  t\right)  }dt,
\end{equation*}
and inserting the identities
$$
F' (r) =
\frac{\int_{0}^{r}\sigma^{m-l-1}\left(  s\right)
ds}{\sigma^{m-l-1}\left(  r\right)  }= I_{m-l-}^{-1}(r)>0\,\,\,
{\rm and}\,\,\,
\
F''(r) = 1 - (m-l-1) \frac {\sigma'}\sigma F'(r),
$$
into the last inequality yields
\begin{eqnarray}
\Delta_M (F\circ \rho\circ\pi \circ \varphi)(x) &\geq &
\left[1 - (m-l-1)F'\frac{\sigma'}\sigma \right](\rho_{x}) \sum_{i=1}^m \langle
\nabla^N \rho_{N}, d\pi_N d\varphi X_i\rangle^2\nonumber \\
&& + \bigl (F'\frac {\sigma'}\sigma\bigr)(\rho_{x})
\sum_{i=1}^m |d\pi_N d\varphi(X_i)|^2
+\,F'(\rho_{x})\langle\nabla^N \rho_{N}, d\pi_N\mathbf{H}\rangle \nonumber.
\end{eqnarray}
Since $d\pi_N$ is the orthogonal projection onto $T_zN$ which is of codimension
$l$ in $T_{(z,u)}(N\times L)$, and $d\varphi(X_i)$ are $m$-orthonormal vectors, it it easy to verify that
\[
\sum_{i=1}^m |d\pi_N d\varphi(X_i)|^2  \geq m-l= (m-l)|\nabla^N\rho_N |^2.
\]
Using this, the fact that $|\nabla^N\rho|^2 = 1\geq \sum_i \langle\nabla^N\rho_N, d\pi_N d\varphi X_i\rangle^2$, and the assumption that $\mathcal{I}_{m-l}$ is nondecreasing in the form
\[
(m-l) \frac{\sigma'}{\sigma} F'(r)\geq 1
\]
we conclude that the right hand side of the above inequality is bounded below by
\begin{eqnarray}
(m-l)F'\frac{\sigma'}\sigma (\rho_{x}) \bigl[ 1-
\sum_{i=1}^m \langle \nabla^N \rho_{N}, d\pi_N d\varphi X_i\rangle^2\bigr]
+ \sum_{i=1}^m \langle \nabla^N \rho_{N}, d\pi_N d\varphi X_i\rangle^2 - F'(\rho_{x})|\mathbf{H}|&&\nonumber\\
\geq 1 - F'(\rho_x) |\mathbf{H}|&&\nonumber
\end{eqnarray}
Thus, if $X=\nabla(F\circ\rho_N \circ\pi_N\circ\varphi)$, then
\[
\mathrm{div}_M X\geq
1 - F'(\rho_x) |\mathbf{H}|,
\]
and the estimate
\[
(1-A)^2 \inf_{[0,r_{\pi_N\varphi} ]}
\frac{I_{m-l}\left(  r\right)
^{2}}{4}
\]
follows arguing as in Theorem~\ref{spec-est}.

In a completely similar manner, the second estimate in a) follows applying Barta's theorem to the function $u=\tilde{F}\circ\rho_N\circ \pi_N\circ \varphi$ with
\[
\tilde{F}(r)= \int_r^{\mathrm{diam}(\pi_N\varphi(M))} I_{m-l}(t)^{-1} dt,
\]
and conclusion b) in the statement is obtained noticing that if $\pi_N \circ \varphi $ is proper, and $I_{m-l}^{-1}$ is integrable, then $-u$ is a a weak maximum principle violating exhaustion function.
\end{proof}

In a similar fashion we have the following analogue of Theorem~\ref{thmMCE}, which complements previous results by
L. J. Alias, G.P. Bessa and M. Dacjzer, \cite{ABD-MathAnn}.

\begin{theorem}
\label{mcurv_product}
Let $\varphi\colon \!M\rightarrow N$ be an isometric
immersion of a stochastically complete, $m$-dimensional Riemannian manifold
$M  $ into the product $N\times L$, where $N$ and $L$ are complete Riemannian manifolds of dimension $n$ and $l$ respectively, with $m\geq l+1$, and $N$ satisifes the conditions listed in the statement of Theorem~\ref{sp-est-products}.
Assume also that
\begin{enumerate}
\item[(i)] $\mathcal{I}_{m-l}\left(  r\right)  $ is non-decreasing
on $[0,\pi_N(\varphi(M))]$,
\item[(ii)]
$
{I_{m-l}\left(  r\right)^{-1}  }\in L^{1}\left(+\infty\right)
$
if $\pi_N \varphi$ is unbounded.
\end{enumerate}
Then
\[
\sup_M \frac{\vert\mathbf{H}(x)\vert}{I_{m-l}(\rho_x)}\geq 1
\]
In particular, if  $r_{\pi_N\varphi}=+\infty$, ($\pi_N \varphi$ is unbounded in $N$) then
\[
\limsup_{x\to \infty} \frac{\vert\mathbf{H}(x)\vert}{I_{m-l}(\rho_x)}\geq 1.
\]
If $r_{\pi_N \varphi}=+\infty$ and $I_{m-l}(r)^{-1}\to 0$ as $r\to +\infty$  then
\[
\sup_M |H| = +\infty.
\]

\end{theorem}
We conclude this section noting that the above arguments can be used to give the following
version for products of the already cited mean exit time comparison results obtained in \cite{markvorsen-JDG}.

\begin{theorem}[Stochastic estimates]
\label{L1Liouville-products}
Let $M$, $N$ and $L$ be complete Riemannian manifolds of dimensions $m,$ $n$ and $l$,
 as in the statement of Theorem~\ref{mcurv_product}, with $m\geq l+1$.
 Let $\varphi: M\to N\times L$ be a minimal immersion, and assume that
(i) and (ii) in the statement of Theorem~\ref{mcurv_product} hold. Then
 $M$ is not $L^{1}$-Liouville.
\end{theorem}

\end{document}